\documentclass[11pt,notitlepage,twoside,a4paper]{amsart}
 \usepackage{amsfonts}
 \usepackage{mathrsfs}
\usepackage{amsmath,amssymb,enumerate}

\usepackage{epsfig,fancyhdr,color}

\usepackage{amssymb}
\usepackage{amsmath,amsthm}
\usepackage{latexsym}
\usepackage{amscd}
\usepackage[latin1]{inputenc}
\usepackage[all]{xy}

\newcommand{\C}{\mathbb C}
\newcommand{\N}{\mathbb N}
\newcommand{\Z}{\mathbb Z}

\newtheorem{theorem}{\rm\bf Theorem}[section]

\newtheorem{lemma}[theorem]{\rm\bf Lemma}
\newtheorem{corollary}[theorem]{\rm\bf Corollary}
\newtheorem*{theorem 1}{\rm\bf Proposition 1}
\newtheorem*{theorem 2}{\rm\bf Proposition 2}
\newtheorem*{proposition 1}{\rm\bf Proposition 3.2}

\theoremstyle{definition}

\theoremstyle{remark}

\def\interieur#1{\mathord{\mathop{\kern 0pt #1}\limits^\circ}}

\definecolor{NoteColor}{rgb}{1,0,0}

\title[Cluster sets of curves]{The initial and terminal cluster sets of an  analytic curve}

\author{P. M. Gauthier}
\address{Paul Gauthier, D\'epartement de math\'ematiques et de statistique, Universit\'e de Montr\'eal,  Pavillon Andr\'e-Aisenstadt, 2920, chemin de la Tour, Montr\'eal (Qu\'ebec)  H3T 1J4 Canada, 
 email: 
 {\rm gauthier@dms.umontreal.ca}}

\date{\today}

\keywords{Analytic curves, cluster sets}
\subjclass{30B40}

\thanks{ This research was supported through the programme ``Research in Pairs"  by the Mathematisches Forschungsinstitut Oberwolfach in 2016 as well as by NSERC (Canada) }

\begin{document}

\maketitle

\begin{abstract}
For an analytic curve $\gamma:(a,b)\rightarrow \C,$ the set of values approached by  $\gamma(t),$ as $t\searrow a$ and as $t\nearrow b$ can be any two continuua of $\C\cup\{\infty\}.$ 
\end{abstract}

\section{Introduction}

For $-\infty\le a<b\le +\infty,$ and a Riemann surface $X,$ we say that $\gamma:(a,b)\rightarrow X$ is a real-analytic curve, if it is  real-analytic for every local coordinate of $X.$ That is, for every $t_0\in(a,b)$ and every local coordinate $z$ at $\gamma(t_0),$ The function $z\circ\gamma$ is representable by a power series  in an interval centered at $t_0.$ Analytic curves in Riemann surfaces were studied in \cite{GN}, however, 
in this note, we consider only the case that $X$ is the Riemann sphere $\overline\C=\C\cup\{\infty\}.$  Thus, $\gamma$ is analytic if and only if $\gamma$ can be developed in a power series in an interval about each $t_0\in (a,b),$ for which $\gamma(t_0)$ is finite and $1/\gamma$ can be developed in a power series about each point $t_0,$ where $\gamma(t_0)=\infty.$ A real-analytic curve $\gamma$ is said to be {\em regular} if its derivative never vanishes, by which we mean that $(z\circ\gamma)^\prime$ has no zeros, for every local coordinate $z.$ For brevity, we shall say (as many authors do) that $\gamma$ is an analytic curve to mean that $\gamma$ is a regular real-analytic curve. We denote by  $C(\gamma,a)$ and $C(\gamma,b)$ respectively the initial and terminal cluster sets:
$$
	C(\gamma,a) = 
	\{w\in\overline\C: \exists \, t_n\in(a,b), \, t_n\rightarrow a, \, \gamma(t_n)\rightarrow w\},
$$
$$
	C(\gamma,b) = 
	\{w\in\overline\C: \exists \, t_n\in(a,b), \, t_n\rightarrow b, \,  \gamma(t_n)\rightarrow w\}.
$$
Both cluster sets  are continua in $\overline\C ,$ that is,  nonempty compact connected sets. A  {\em degenerate} continuum is a continuum consisting of a single point. 
Our principal result is that the initial and terminal cluster sets can be arbitrarily prescribed continua in $\overline\C.$ For characterizations of other cluster sets, see Theorems 4.3-4.5 and also page 165 in  \cite{CL}.

It is also of interest to know whether an analytic curve can be extended in some sense (which we now specify). A notion of extendability for an analytic curve was introduced by Nestoridis and Papadopoulos in \cite{NP}. 
Let us say that an analytic curve $\gamma(t), \, a<t<b,$ can be extended {\em initially}, if there is an analytic curve  $\sigma(s), \, L<s<R,$ a value $A\in (L,R),$ and an analytic change of parameter $t:(A,R)\rightarrow(a,b),$ such that $\sigma(s)=\gamma(t(s)),$ for $A<s<R.$  We say that $\sigma$ is an initial analytic extension of $\gamma.$ A terminal analytic extension is defined analogously. Let us say that an analytic curve is maximal as an analytic curve (or analytically maximal) if it has neither an intiial nor a terminal analytic extension. 

It is a pleasure to thank Vassili Nestoridis for suggesting I write this paper and for helpful conversations. This manuscript has appeared as an Oberwolfach Preprint
(OWP 2016-25).


\section{Results and preparatory lemmas}

\begin{theorem}\label{entire}
For any two continua $K^-$ and $K^+$ of the Riemann sphere, there exists an analytic curve $\gamma:(-\infty,+\infty)\rightarrow\C,$ which is the restriction of an entire function, such that  
$$
	C(\gamma,-\infty)= K^-
$$
and
$$
	C(\gamma,+\infty) =K^+.
$$ 
Moreover, the curve $\gamma$ is  maximal as an analtyic curve. 
\end{theorem}

For distinct points $z_1$ and $z_2$ in $\C,$ we denote by $[z_1,z_2]$ the line segment from $z_1$ to $z_2.$ 
Also, we denote by $[+i,+i\infty)$ and  $[-i,-i\infty)$ the closed half-lines $\{z=x+iy:x=0, \, +1\le y<+\infty\}$ and $\{z=x+iy:x=0, \, -\infty<y\le -1\}$ respectively.

\begin{corollary}\label{analytic}
For any two continua $K^-$ and $K^+$ of the Riemann sphere, there exists an analytic curve $g:(-\infty,+\infty)\rightarrow\C,$ which extends to a (locally) conformal mapping $G$ on the doubly-slit plane 
$$
	\C\setminus \big([+i,+i\infty)\cup[-i,-i\infty)\big),
$$ 
for which  
$$
	C(g,-\infty)=C(G,-\infty) = K^-
$$
and
$$
	C(g,+\infty)=C(G,+\infty) =K^+.
$$
Moreover, the curve $g$ is maximal as an analtyic curve. 
\end{corollary}

\begin{lemma}\label{corner} For $t_1<t_2<t_3,$ and non colinear points $z_1,z_2,z_3\in\C,$  consider the parametrisations $\sigma_1:[t_j,t_{j+1}]\rightarrow [z_j,z_{j+1}]$ of the segments $[z_j,z_{j+1}],$ given by 
$$
	\sigma_j(t)=z_j+\frac{t-t_j}{t_{j+1}-t_j}(z_{j+1}-z_j), \quad \mbox{for}  \quad   j=1,2.
$$
For each  $\epsilon>0,$ and all sufficiently small $delta>0,$ there is a $C^1$-smooth curve  $\sigma_\delta:[t_1,t_3]\rightarrow\C,$ with nonvanishing derivative, such that 
$$
	\sigma_\delta(t) = \left\{ 
\begin{array}{ll}
	\sigma_1(t) 	&	\mbox{if \, $t_1\le t\le t_2-\delta$};\\
	\sigma_2(t)	&	\mbox{if \, $t_2+\delta<t\le t_3$};
\end{array}
\right.
$$
$\sigma_\delta^\prime(t)\not=0,$ for $ t\in [t_1,t_3];$ and $|\sigma_\delta(t)-\sigma_j(t)|<\epsilon,$ for $t\in [t_j,t_{j+1}], \, j=1,2.$
\end{lemma}

\begin{proof}
By  linear motions, we may assume $t_1=-1, t_2=0, t_3=1, z_1=-1+ic, z_2=0, z_3=1+ic$ 
and $\sigma_1(t)=t-ict, \sigma_2(t)=t+ict.$ 
We can smooth on the interval $[-1,1],$ by taking the even function $\sigma_\delta,$ which for $0\le t\le1$ is defined by 
$$
	\sigma_\delta(t) = t+ic\left(t^2+\delta(2-\delta)\right)/2.
$$
Since
$$
	|\sigma_\delta(t)-\sigma_j(t)| \le c\left(\delta^2+\delta(2-\delta)\right)/2 +c\delta,
$$
we indeed have  $|\sigma_\delta(t)-\sigma_j(t)|<\epsilon,$ for all sufficiently small $\delta.$
\end{proof}

A different proof of the lemma can be given by constructing a circle $C$ tangent to the segments $[z_2,z_1,]$ and $[z_2,z_3],$ whose center lies on the bisector of the acute angle formed by  these segments.   Denote by $w_1$ and $w_2$ the points of tangency and replace the two segments $[w_1,z_2]$ and $[z_2,w_2]$ by the smaller of the two arcs of $C\setminus \{w_1,w_2\}.$ We form a curve without corners: the concatenation of the segment $[z_1,w_1],$ the circular arc from $w_1$ to $w_2$ and the segment $[w_1,z_3].$ With an appropriate parametrisation of the arc, this curve is analytic. By choosing the center close to the vertex $z_2,$ we can make this curve as close to the original polygonal curve as we wish.


\smallskip
\noindent
{\bf Step 1: A sequence with prescribed cluster set}

Let $K$ be a  continuum in $\overline\C.$    Of course, it is easy to construct a sequence in $\C$ with $K$ as cluster set, but we wish this sequence to have special properties. 
We begin with the following well-known fact.

\begin{lemma}\label{metric}
Let $K$ be a connected metric space. Let $\delta>0$ and $p,q\in K.$ 
Then, there exists $n\in\N$ and $\{p_1,p_2,\dots,p_n\in K\},$ with $p_1=p, \, p_n=q$ and \mbox{dist}$(p_j,p_{j+1})<\delta, \, j=1,\ldots,n-1.$ 
\end{lemma}

Suppose first that $K$ is a nondegenerate continuum in $\overline\C.$
Let $z_j\in K\cap\C, j=1,2,\ldots,$ be a dense sequence of distinct points in $K.$ By Lemma \ref{metric}, and by induction, there is an increasing sequence $n(j)\in \N$ and a sequence $p_n\in K,$ such that 
$p_{n(j)}=z_j$ and 
$$
	|p_n-p_{n+1}|<\frac{1}{j}, \quad  \mbox{for} \quad  \mbox  n(j)\le n<n(j+1). 
$$
By  inserting nearby points (possibly not in $K,$) we may assume that no three consecutive points are colinear. Moreover, by occasionaly inserting at most two nearby points, we may assume that that there are  subsequences   $p_{n(k)}$ and $p_{n(\ell)},$  both of which approach every point of $K,$ such that every segment $[p_{n(k)},p_{n(k)+1}]$ is horizontal (with $[p_{n(k)}$ as left end point  and every segment  $[p_{n(\ell)},p_{n(\ell)+1}]$ is vertical (with  $p_{n(\ell)}$ as lower point). 

Now, suppose $K$ is a degenerate continnum in $\overline\C$ (that is, a point). We repeat the above procedure, where now we begin with an arbitrary sequence $z_j$ of distinct points in $\C,$ which converges to $K.$ 

We recapitulate this construction in the following lemma.  

\begin{lemma}\label{sequence}
For any two continua $K^-$ and $K^+$ of the Riemann sphere, there exists a double  sequence $\{p_n, n\in \Z\}$ in $\C,$ such that the cluster set of the sequence $p_0,p_{-1},\ldots,$ is precisely $K^-$ and the the cluster set of the sequence $p_0,p_1,\cdots,$ is precisely $K^+.$ 
No three consecutive points are colinear. There are subsequences $p_{n(i)}$ and $p_{n(j)}$ of  $p_0,p_{-1},\ldots$ both of which approach every point of $K^-$ such that every segment $[p_{n(i)},p_{n(i)+1}]$ is horizontal and every segment  $[p_{n(j)},p_{n(j)+1}]$ is vertical. Similarly, There are subsequences $p_{n(k)}$ and $p_{n(\ell)}$ of  $p_0,p_1,\ldots$ both of which approach every point of $K^+$ such that every segment $[p_{n(k)},p_{n(k)+1}]$ is horizontal (with $p_{n(k)}$ as left end point) and every segment  $[p_{n(\ell)},p_{n(\ell)+1}]$ is vertical (with  $p_{n(\ell)}$ as lower point). 
\end{lemma}


\smallskip
 \noindent
{\bf Step 2:\\A polygonal curve with prescribed initial and terminal cluster sets}
  
Suppose $\eta_n$ is a linear mappig of $[n,n+1]$ onto the segment $[p_n,p_{n+1}]$
We define a curve $\eta:(-\infty,+\infty)\rightarrow\C$ by setting $\eta=\sum_{n=-\infty}^\infty\eta_n,$ where the sum represents concatination. Such a curve is said to be a  polygonal curve  with nodes $p_n,$

\begin{lemma}\label{polygonal}
For any two continua $K^-$ and $K^+$ of the Riemann sphere, there exists a polygonal curve $\eta:(-\infty,+\infty)\rightarrow\C,$  for which  
$$
	C(\eta,-\infty)= K^- \quad  \mbox{and} \quad C(\eta,+\infty) =K^+.
$$
No three consecutive nodes are colinear.  There are sequences $s_{n(i)}$ and $s_{n(j)}$ of real numbers tending to $-\infty,$ such that at these values  $\eta$ has a non-vanishing derivative, with  $\arg(\eta^\prime(s_{n(i}))=0$ and $\arg(\eta^\prime(s_{n(j}))=\pi/2.$ Moreover, the sequences $\eta(s_{n(i)})$ and $\eta(s_{n(j)})$ have $K^-$ as set of limits. There are analogous sequences  $s_{n(k)}$ and $s_{n(\ell)}$  with respect to $K^+.$
\end{lemma}

\begin{proof}
By Lemma \ref{sequence} there is a sequence $p_n; n=0,1,2,\ldots,$ associated to $K^+$ and a sequence $p_n; n=0,-1,-2,\ldots,$ associated to $K^-.$ Let $\eta_n$ be a linear mapping of the interval $[n,n+1]$ onto the segment  $[p_n,p_{n+1}]$ and put $\eta=\sum\eta_n.$ 
Since the lengths of the segments  $[p_n,p_{n+1}]$ tend to $0,$ as $n\rightarrow\infty,$ it follows that the cluster set of $\eta$ at $-\infty$ is $K^-$ and the cluster set of $\eta$ at $+\infty$ is $K^+.$ 

We construct the sequence $s_{n(i)}$ as follows. Let $\{n(i)\}$ be the sequence from Lemma \ref{sequence}. From the previous paragraph, $\eta_{n(i)}$ is a linear mapping of the interval $[n(i),n(i+1)]$ onto the horizontal segment  $[p_{n(i)},p_{n(i)+1}].$ As $s_{n(i)}$  we  choose the mid-point of the open interval $(n(i),n(i+1)).$ Clearly, the sequence  $s_{n(i)}$ has the required properties. The other three sequences  $s_{n(j)}, s_{n(k)}$ and $s_{n(\ell)}$  are constructed similarly. 
\end{proof}


\smallskip
 \noindent
{\bf Step 3: A smooth curve with prescribed initial and terminal cluster sets}

Now we shall smooth the polygonal curve $\eta.$ 

\begin{lemma}\label{smooth}
For any two continua $K^-$ and $K^+$ of the Riemann sphere, there exists a smooth curve $\sigma:(-\infty,+\infty)\rightarrow\C,$  for which  
$$
	C(\sigma,-\infty)= K^- \quad  \mbox{and} \quad C(\sigma,+\infty) =K^+.
$$
The curve $\gamma$ has the same values as the polygonal curve $\eta$ in a neighborhood of the values  $s_{n(i)}, s_{n(j)}, s_{n(k)}$ and $s_{n(\ell)}.$
\end{lemma}

\begin{proof}
We begin with the polygonal curve $\eta=\sum\eta_n$ from Lemma \ref{polygonal}. We replace each $\eta_n$ by a smoothing $\sigma_n$ of $\eta_n$ obtained by Lemma \ref{corner}, such that
$$ 
	|\sigma_n(t)-\eta_n(t)|<|n+1|^{-1}, \quad \mbox{for} \quad t\in[n,n+1],  \quad  n\in\Z.
$$
The concatination $\sigma=\sum\sigma_n$ has the required properties. Indeed, it
 has the required initial and terminal cluster sets, because $|\sigma(t)-\eta(t)|\rightarrow 0,$ as $t\rightarrow\infty.$  Each time we invoke Lemma \ref{corner}, we may choose $\delta$ so small that $\sigma(t)=\eta(t)$ in an interval about the mid-point of the parameter interval $(n,n+1).$ The four sequences consist of such mid-points.  
\end{proof}


\section{Proof of Theorem \ref{entire}}

By a theorem of Hoischen \cite{H} (see also \cite[Cor. 1.4]{GK} for an elementary proof), for each $C^1$-smooth function $\sigma:(-\infty,+\infty)\rightarrow\C$ and each continuous function $\epsilon:(-\infty,+\infty)\rightarrow (0,+\infty),$ there is an entire function $f$ such that
$$
	|f^{(j)}(t)-\sigma^{(j)}(t)| < \epsilon(t), \quad \mbox{for all} \quad t\in(-\infty,+\infty), \quad j=1,2. 
$$
If $\sigma$ is the curve from Lemma \ref{smooth}  and the function $\epsilon$  tends to zero, as $t\rightarrow \infty,$ then
the curve $\gamma(t)=f(t),$ for $t\in(-\infty,+\infty),$ has the same initial and terminal cluster sets as
the curve $\sigma.$ That is, $\gamma$ has $K^-$ and $K^+$ as initial and terminal cluster sets. 

Moreover,  we claim that the curve $\gamma$ cannot be extended to $-\infty$ or $+\infty$ analytically  for any reparametrization of the increasing parameter, provided we choose $\epsilon$ to decrease sufficiently rapidly. In fact, this is obvious in the case that the corresponding cluster set is non-degenerate. The following proof is thus only of interest if one or both of the initial and terminal cluster sets are degenerate continua (singletons).

Since $\sigma^\prime(t)\not=0,$ it follows that $\arg\sigma^\prime$ is uniformly continuous on compact subsets of $(-\infty,+\infty).$ Hence we may choose $\epsilon$ to decrease so rapidly that 
$\arg\gamma^\prime(t)$ is close to zero for $t=s_{n(i)}$ and $t= s_{n(k)}$ and is close to $\pi/2$ for $t=s_{n(j)}$ and $t=s_{n(\ell)}.$ 
It follows that $\arg\gamma^\prime(t)$ diverges as $t\rightarrow-\infty$ and as $t\rightarrow+\infty.$ Consequently, $\gamma$ cannot be extended analytically to any larger Riemann surface by any analytic reparametrization, for such an extension would have to be conformal and (by definition) preserve angles. In particular, if $K^-$ or $K^+$ is a point $P$ on the Riemann sphere, then $\gamma$ cannot be extended analytically through $P$ by any analytic reparametrisation. 

This concludes the proof of Theorem \ref{entire}.


\section{Proof of the corollary}

Let $f$ and $\gamma$ be the entire function and analytic curve obtained from Theorem \ref{entire}, where $\gamma$ is the restriction of $f$ to the real line. Let $\Omega$ be a neighborhood of the real line, in which $f^\prime$ is zero-free. We may assume that $\Omega$ has the form of a ``strip" $w=u+iv: |v|<\varphi(u).$  We may assume that $\varphi$ decreases to zero so rapidly that $|f(u+iv)-\gamma(u)|<1/(1+u)$ for $|v|<\varphi(u).$ This assures us that $f$ has the same initial and terminal cluster sets in the strip $\Omega$ as $\gamma$ has on the real line.

Let $h$ be the conformal mapping of $\C\setminus \big([+i,+i\infty)\cup[-i,-i\infty)\big)$ onto the strip $\Omega,$ which sends $-\infty$ to $-\infty,$ $+\infty$ to $+\infty,$ $0$ to $0$ and the real line to itself. The locally conformal function $G=f\circ h$  and its restriction $g$ to the real line, have the required properties. Indeed, since  $h$ is an order preserving homeomorphism of the real line, $g=\gamma\circ h$ has the same initial and terminal cluster sets as $\gamma.$  Similarly, $G=f\circ h$ has the same initial and terminal cluster sets as $f$ in $\Omega,$ which are the same initial and terminal cluster sets as those of $\gamma.$ 

There remains to check that $g$ cannot be analytically extended. We note that $g^\prime(u)=\gamma^\prime(h(u))h^\prime(u),$ and $h^\prime(u)$ is real and positive, so $\arg h^\prime(u)=0.$ Thus, 
$\arg g^\prime(u)=\arg \gamma^\prime(h(u))+\arg h^\prime(u)=\arg \gamma^\prime(h(u)).$ Since $\arg\gamma^\prime(t)$ diverges as $t\rightarrow-\infty$ and as $t\rightarrow+\infty$
the same holds for $\arg g^\prime(u),$ as  $u\rightarrow-\infty$ and as $u\rightarrow+\infty.$  
Thus, $g$ cannot be extended analytically and this concludes the proof of the corollary.


\section{Examples of maximal analytic curves}

If the initial cluster set of a curve is a singleton $\{P\}$, we call $P$ the initial end of the curve. Similarly, if the terminal cluster set of a curve is a singleton, we call it the terminal end of the curve. A particular case of Theorem \ref{analytic} is that, for any two points (not necessarily distinct) $k^-$ and $k^+$ of the Riemann sphere, there is a maximal analytic curve, having  $k^-$ and $k^+$ as initial and terminal  ends respectively. We now give a few explicit examples of maximal analytic curves having both initial and terminal ends. As in the general case, proved above, the reason that these curves are maximal is that the argument of the  tangent $\gamma^\prime(t)$ diverges as $t\rightarrow\pm\infty.$

Example 1. Both ends are finite and equal.

$$
	\gamma(t) = e^{-t^2+it}, \quad  -\infty<t<+\infty. 
$$

Example 2. Both ends are finite and distinct.  Consider the function 
$$
	\psi(s) = s\exp\left(\frac{1}{1-s^2}\right), \quad  -1<s<+1. 
$$
The function $\psi$ is analytic with positive derivative and hence has an analytic inverse. $\eta:(-\infty,+\infty)\rightarrow (-1,+1).$ The analytic curve 
$$
	\gamma(t)=\eta(t)+i(\eta^2(t)-1)\sin\left(\exp\left(-(\eta^2(t)-1)^{-1}\right)\right) 
$$ 
has $\pm 1$ as ends.  As $x\searrow 0,$ $\exp(-x^{-1})$ approaches $0$ much faster and so the argument of $\gamma^\prime(t),$ does not have a limit, as $t\rightarrow\pm 1.$

Example 3. One end is finite and one is infinite.
$$
	\gamma(t) = e^{t+it}, \quad  -\infty<t<+\infty.
$$

Example 4. Both ends are infinite. 
$$
	\gamma(t) = e^{t^2+it}, \quad  -\infty<t<+\infty.
$$




\begin{thebibliography}{1}
 
\bibitem{CL}  Collingwood, E. F.; Lohwater, A. J. The theory of cluster sets. Cambridge Tracts in Mathematics and Mathematical Physics, No. {\bf 56} Cambridge University Press, Cambridge 1966. 

\bibitem{GK} Gauthier, P. M.; Kienzle, J.
Approximation of a function and its derivatives by entire functions. 
Canad. Math. Bull. {\bf 59} (2016), no. 1, 87-94. 

\bibitem{GN} Gauthier, P. M.; Nestoridis, V. Conformal extensions of functions defined on arbitrary subsets of Riemann surfaces. Arch. Math. (Basel) {\bf 104} (2015), no. 1, 61-67. 

\bibitem{H} Hoischen, L.
Approximation und Interpolation durch ganze Funktionen. (German)
J. Approximation Theory {\bf 15} (1975), no. 2, 116-123.

\bibitem{NP} Nestoridis, V.; Papadopoulos, A.;
Arc length as a global conformal parameter for analytic curves.
J. Math. Anal. Appl. {\bf 445} (2017), no. 2, 1505-1515. 

 
\end{thebibliography}
\end{document}